\documentclass[11pt]{article}
\usepackage[numbers, sort]{natbib}
\usepackage{amsmath}
\usepackage{amsfonts}
\usepackage{amssymb}
\usepackage{epsfig}
\usepackage{amscd}
\usepackage{latexsym}
\usepackage[colorlinks=black,linkcolor=red,anchorcolor=blue,citecolor=green]{hyperref}
\usepackage{CJK}
\usepackage{color}
\usepackage{subfigure}

\newtheorem{theorem}{\bf Theorem}[section]

\newtheorem{lemma}[theorem]{\bf Lemma}

\newtheorem{corollary}[theorem]{\bf Corollary}

\newenvironment{proof}{\noindent{\em Proof:}}{\quad \hfill$\Box$\vspace{2ex}}

\DeclareMathOperator{\Span}{span}
\DeclareMathOperator{\diag}{diag}

\def \aP {\mathbb P}

\def \aS {\mathbb S}

\def \bO {\mathcal{O}}

\def \bS {\mathcal{S}}

\def \Bb {\textbf{b}}

\def \Ba {\boldsymbol{a}}
\def \Bb {\textbf{b}}
\def \Bx {\textbf{x}}
\def \Bd {\textbf{d}}

\def \Bk {\boldsymbol{k}}

\def \Bu {\textbf{u}}

\pdfoutput=1
\begin{document}
\begin{CJK}{GBK}{song}

\title{\bf Fast non-polynomial interpolation and integration for functions with logarithmic singularities}
\author{ Yinkun Wang
         \footnotemark[3],
         Xiangling Chen
         \footnotemark[4],
         Ying Li
         \footnotemark[3],
         and
         Luo Jianshu
         \footnotemark[3]
         }

\renewcommand{\thefootnote}{\fnsymbol{footnote}}

\footnotetext[3]{College of Science, National University of Defense Technology, Changsha 410073, People's Republic of China. }
\footnotetext[4]{Key Laboratory of High Performance Computing and Stochastic Information Processing(HPCSIP) (Ministry of Education of China), College of Mathematics and Computer Science, Hunan Normal University, Changsha, Hunan 410081, P. R. China
.}
\date{Apr. 17, 2018}
\maketitle{}
\begin{abstract}
A fast non-polynomial interpolation is proposed in this paper for functions with logarithmic singularities. It can be executed fast with the discrete cosine transform.
Based on this interpolation, a new quadrature is proposed for a kind of logarithmically singular integrals.
The interpolation and integration errors are also analyzed.
Numerical examples of the interpolation and integration are shown to
 validate the efficiency of the proposed new interpolation and the new quadrature. 
\end{abstract}


\section{Introduction}
Integral equations with singular kernels are frequently encountered in different branches of mathematical physics
in the formulation of initial value and boundary value problems. The logarithmic singularity appears typically in the kernel of the integral equations, especially in two spatial dimensions in the electromagnetics and acoustics. For instance, the Helmholtz equation in two dimensions can be reformulated into an integral equation with its kernel expressed in the form of Hankel function of order zero which has a logarithmic singularity \cite{COLTON2013}. Another famous example in the analysis of thin wire antennas is  the Hall\'{e}n's integral equation whose kernel is also logarithmically singular \cite{2007BRUNO,WANG2015wire}.

In the numerical solution (such as in Galerkin solution and collocation solution) of integral equations with the logarithmic singularity, one is always led naturally to the evaluation of functions with logarithmic singularity and the computation of integrals with the logarithmic singularity.
The development of the fast and accurate interpolation of functions with the logarithmic singularity can provide a way to evaluate the functions quickly and to compute the singular integrals accurately.
For this purpose, we aim to propose a fast non-polynomial interpolation based on Chebyshev polynomials for functions of the form,
\begin{equation}\label{e1}
K(x):=g_1(x)+g_2(x)\log|x-\alpha|, x\in [-1,1]\backslash\{\alpha\}
\end{equation}
where $\alpha\in[-1,1]$. The functions $g_1$ and $g_2$ have the high function regularity or even are smooth, but they may be mixed together in an expression.

Due to the logarithmic singularity, classical polynomial interpolations, such as the Lagrange interpolation and the Chebyshev interpolation, shall lead to a great error in approximation of singular functions of this kind.
It is because of the singular factor of the high-order derivative of the functions in the remainder of the polynomial interpolation.
To handle this difficulty, the interpolation was made on graded meshes with more meshes near the singularity and less faraway \cite{RICE1969}.
This technique is frequently used in the numerical solution of singular equations and the calculation of singular integrals. A weighted Gaussian quadrature rule was proposed in \cite{CROW1993}.
The Clenshaw-Curtis quadrature was also used for the computation of singular integrals due to its efficient performance \cite{WANGH2018,HASEGAWA2018}.
Besides, many researches have been done on the oscillatory integrals with the logarithmic singularity  \cite{DOMINGUEZ2014,CHENR2015,KANGH2015,KANGH2017,KANGH2018}.

Different from graded meshes, our basic idea for the new interpolation is to append singular basis functions to the Chebyshev interpolation basis.
This enrichment shall capture the singularity of the singular functions and thus reduce the effect of the singularity on the accuracy. This similar idea has been used in the numerical solution of Fredholm and Volterra integral equations \cite{CAO1994,CAO2003}.
Combining the FFT, we then develop a fast scheme to execute this non-polynomial interpolation. Based on this interpolation, we propose a new quadrature for a kind of logarithmically singular integrals.

This paper is organized as follows. The fast non-polynomial interpolation is proposed in section 2 and then analyze its computation cost and interpolation error. In section 3, we present the integration based on the new interpolation. In section 4, numerical examples are given to show the efficiency of the proposed interpolation for singular functions and that of the integration for singular integrals.

\section{Non-polynomial interpolation for singular functions}
In this section, we present the non-polynomial interpolation for singular functions of the form \eqref{e1} based on the Chebyshev polynomials of the first kind using the nodes of Chebyshev points, i.e. the zeros of Chebyshev polynomials of the first kind. An implementation using the discrete cosine transform (DCT) is proposed for the interpolation at the meantime. The third part of this section is about the error analysis for this interpolation.

\subsection{Derivation of interpolation}
We consider the problem of approximating a function $K$ of the form \eqref{e1} by a certain combination of
singular functions and polynomial functions. More precisely, we replace $K$ by a function $K_{n_1,n_2}$ of
the form
\begin{equation}\label{e2}
K_{n_1,n_2}(x):=\sum_{j=0}^{n_1-1}a_{j}T_j(x)+\log|x-\alpha|\sum_{k=0}^{n_2-1}b_kT_k(x), x\in [-1,1]\backslash \{\alpha\}
\end{equation}
where $T_j(x):=\cos \left(j\arccos x\right)$ are Chebyshev polynomials of the first kind, $n_1, n_2$ are positive integers and $a_j,b_k$ are the unknown coefficients to be determined. Let $n:=n_1+n_2$. In order to derive the coefficients, the error term defined as
\[
E_{n_1,n_2}(K,x):=K(x)-K_{n_1,n_2}(x)
\]
is required to vanish at Chebyshev points
\[
x_j=\cos\frac{(2j+1)\pi}{2n}, j=0,1,\ldots,n-1.
\]
We note that $n$ should be selected such that $\alpha$ is different from any Chebyshev points $x_j$ we used here since $\alpha$ is a singular point.

We next present the interpolation in the matrix form. Let $\Bx:=[x_j: j=0,\ldots,n-1]^T$, $\Ba:=[a_j: j=0,\ldots,n_1-1]^T$ and $\Bb:=[b_j: j=0,\ldots,n_2-1]^T$ where the superscript $T$ means the transpose. We then define the Vandermonde-type matrix $A:=[T_0(\Bx)|T_1(\Bx)|\ldots|T_{n-1}(\Bx)]$. Due to the discrete orthogonal property of Chebyshev polynomials, it is well-known that
\begin{equation}\label{e3}
A^TA=D
\end{equation}
where $D:=\diag(\Bd)$ and $\Bd:=[d_j:j=0,\ldots,n-1]^T$ with $d_0=n$ and $d_j=n/2, j=1,\ldots,n-1$. We also define a matrix
$P:=\diag(\log|\Bx-\boldsymbol{\alpha}|)$
 and a vector $\Bk:=K(\Bx)$,
 where $\boldsymbol{\alpha}$ is the constant vector of $\alpha$ of the same size as $\Bx$.

The non-polynomial interpolation reads in the matrix form below
\begin{equation}\label{e4}
A_{1:n_1}\Ba+PA_{1:n_2}\Bb=\Bk
\end{equation}
where $A_{u:v}$ means the matrix composing the columns of matrix $A$ from the $u$-th column to the $v$-th column, $u,v$ are positive integers.
To solve the linear system \eqref{e4}, we multiply both the sides of the above equation by the matrix $(A_{n_1+1:n})^T$ and obtain that
\begin{equation}\label{e5}
(A_{n_1+1:n})^TPA_{1:n_2}\Bb=(A_{n_1+1:n})^T\Bk
\end{equation}
where we have used the discrete orthogonal property \eqref{e3}.

Suppose that the matrix product $(A_{n_1+1:n})^TP A_{1:n_2}$ is invertible, the unknown $\Bb$ and $\Ba$ can be solved through \eqref{e5} and \eqref{e4} subsequently. This assumption can be verified before using this non-polynomial interpolation since it is independent of the function to be interpolated. Hence we always assume that the matrix product $(A_{n_1+1:n})^TPA_{1:n_2}$ is invertible.

Formula \eqref{e2} with the coefficients $\Ba$ and $\Bb$ forms the non-polynomial interpolation for logarithmically singular functions.

\subsection{Computational complexity}

 The matrix-vector product $A^T\Bu$ computes $\tilde{\Bu}:=[\tilde{u}_0,\tilde{u}_1,\ldots,\tilde{u}_{n-1}]^T$ by the discrete cosine transform of type II (shortened as DCT II) of $\Bu$, that is
 \[
 \tilde{u}_k=\sum_{j=0}^{n-1}u_j \cos\frac{k(2j+1)\pi}{2n},\; k=0,1,\ldots,n-1.
 \]
  where $\Bu:=[u_0,u_1,\ldots,u_{n-1}]^T$ is an arbitrary $n$-vector. This matrix-vector product can be evaluated fast in $\bO(n\log n)$ operations instead of $\bO(n^2)$. Similarly, the matrix-vector product $A\tilde{\Bu}$ produces $\Bu$ by the discrete cosine transform of type III (shortened as DCT III) of $\tilde{\Bu}$ and also can be computed efficiently with $\bO(n\log n)$ operations.

We then present the analysis of the computational complexity of solving \eqref{e5} by GMRES. For iterative methods, the key computation is the matrix-vector product. For this reason, we analyze the computational complexity of $(A_{n_1+1:n})^TPA_{1:n_2}\Bu$ where $\Bu$ is a $n_2$-vector. Let $\hat{\Bu}$ denote a $n$-vector by padding $n_1$ zeros at the end of $\Bu$. Then we have $A_{1:n_2}\Bu=A\hat{\Bu}$. The result of the matrix-vector product $(A_{n_1+1:n})^TPA_{1:n_2}\Bu$ is the last $n_2$ elements of the matrix-vector product $A^TPA\hat{\Bu}$. Since $P$ is a diagonal matrix and the matrix-vector product of $A$ and $A^T$ can be calculated fast within $\bO(n\log n)$ by the DCT III and the DCT II, respectively, the cost for  $A^TPA\hat{\Bu}$ is about $\bO(n\log n)$. Hence a total computation complexity of GMRES in solving \eqref{e5} depends on the iteration number and its most complexity is about $\bO(n_2n\log n)$.


 The coefficients $\Ba$ can be obtained from \eqref{e4} fast within $\bO(n\log n)$. To show this, define two new $n$-vectors
 $\hat{\Ba}:=[a_0,\ldots,a_{n_1-1},0,\ldots,0]^T$ and $\hat{\Bb}:=[b_0,\ldots,b_{n_2-1},0,\ldots,0]^T$ by padding $\Ba$ with $n_2$ zeros and padding $\Bb$ with $n_1$ zeros, respectively. It is obtained obviously from \eqref{e4} that
\[
\hat{\Ba}=D^{-1}A^T(\Bk-PA\hat{\Bb}).
\]
Due the fast computation of the matrix-vector product $A\hat{\Bb}$ and $A^T(\Bk-PA\hat{\Bb})$ within $\bO(n\log n)$, the computational complexity to get $\Ba$ is about $\bO(n\log n)$.

To sum up, the computational complexity for the non-polynomial interpolation is about $\bO(n_2n\log n)$.

\subsection{Error Analysis}
In this part, we discuss the non-polynomial interpolation error for the approximation of singular functions.

We first define the corresponding operator for the non-polynomial interpolation. For this purpose, let $I:=[-1,1]\backslash \{\alpha\}$ where $\alpha\in[-1,1]$. By $\aP_{n_1}$ we denote the space of polynomials of degree less than $n_1$. We define another space consisting of polynomials and some extra logarithmic functions by
\[
\aS^\alpha_{n_1,n_2}(I):=\Span\{x^j,x^k\log|x-\alpha|: j=0,1,\ldots,n_1-1,\; k=0,1,\ldots,n_2-1,\; x\in I\}.
\]
It is clear that $\aP_{n_1}=\aS^\alpha_{n_1,0}$. With the interpolation points $\{x_j: j=0,1,\ldots,n-1 \}\subset I$, $n=n_1+n_2$, by
\[
\bS^\alpha_{n_1,n_2}: C(I)\rightarrow \aS^\alpha_{n_1,n_2}(I)
\]
we denote the interpolation operator that maps the function $f\in C(I)$ onto a function $\bS^\alpha_{n_1,n_2}f\in \aS^\alpha_{n_1,n_2}(I)$ with the property
\[
\left(\bS^\alpha_{n_1,n_2}f\right)(x_j)=f(x_j),\; j=0,1,\ldots,n-1.
\]
The operator $\bS^\alpha_{n_1,n_2}$ is linear and is a projection i.e. $\left(\bS^\alpha_{n_1,n_2}\right)^2=\bS^\alpha_{n_1,n_2}$ since $\bS^\alpha_{n_1,n_2}f=f$ for all $f\in\aS^\alpha_{n_1,n_2}(I)$.
When $n_2=0$, the operator $\bS^\alpha_{n_1,0}$ is degenerated to the normal polynomial interpolation operator,  shorten as $\bS_{n_1,0}$. When $n_2>0$, the interpolant $\bS^\alpha_{n_1,n_2}f$ is a non-polynomial function with the logarithmic singularity.

For the purpose of bounding the interpolation error of the non-polynomial interpolation, we introduce a norm for the operator $\bS_{n_1,n_2}^\alpha$. Define
 \[
 \|\bS_{n_1,n_2}^{\alpha}\|_{1,\infty}:=\sup_{f\in C(I),\|f\|_\infty\neq 0}\frac{\|\bS_{n_1,n_2}^{\alpha}f\|_1} {\|f\|_\infty}
 \]
where $\|u\|_1:=\int_{-1}^1 |u(t)|dt$ for $u\in L^1(I)$ and $\|\cdot\|_\infty$ is the maximum norm.
If the maximum norm of $f\in C(I)$ is unbounded, $\frac{1}{\|f\|_\infty}$ is set to be 0.
We give a bound for this operator norm.
Suppose that the functions $\ell_j, j=0,1,\ldots,n-1$ form a Lagrange-type basis for space $\aS_{n_1,n_2}^\alpha(I)$
satisfying
\[
\ell_j(x_i)=\delta_{i,j},
\]
where $\delta_{i,j}$ denotes the Kronecker symbol. Since $\aS_{n_1,n_2}^\alpha(I)\subset L^1(I)$, there exists that $\|\ell_j\|_1<\infty$, $j=0,1,\ldots,n-1$.

\begin{lemma}\label{onorm}
If the operator $\bS_{n_1,n_2}^{\alpha}$ is the non-polynomial interpolation operator based on the space $\bS_{n_1,n_2}^\alpha(I)$ and points $\{x_j: j=0,1,\ldots,n-1 \}\subset I$, then
    \[\|\bS_{n_1,n_2}^{\alpha}\|_{1,\infty}\leq \sum_{j=0}^{n-1} \|\ell_j\|_1 \]
\end{lemma}
\begin{proof} For $f\in C(I)$,
    \[
    \bS_{n_1,n_2}^{\alpha}f=\sum_{j=0}^{n-1}f(x_j)\ell_j.
    \]
    Together with the triangle inequality of the norm, we have from the above equation that
    \[
    \|\bS_{n_1,n_2}^{\alpha}f\|_1\leq\sum_{j=0}^{n-1}|f(x_j)|\|\ell_j\|_1\leq \|f\|_\infty \sum_{j=0}^{n-1}\|\ell_j\|_1
    \]
    The conclusion follows directly.
\end{proof}

We now present the result of the interpolation error for the non-polynomial interpolation.
\begin{theorem}\label{Inte-E1}
    If $\bS_{n_1,n_2}^\alpha$ is defined based on interpolation points $\{x_j: j=0,1,\ldots,n-1 \}\subset I$, $n=n_1+n_2$, then the interpolation error of $\bS_{n_1,n_2}^\alpha$, for functions of the form
    $K(x)=g_1(x)+g_2(x)\log|x-\alpha|$ with $g_1$ and $g_2$ smooth, satisfies
    \begin{equation}
        \| K -\bS_{n_1,n_2}^\alpha  K \|_1\leq 2(1+\|\bS_{n_1,n_2}^\alpha\|_{1,\infty})\|\tilde{g}_1-\bS_{n_1,0}\tilde{g}_1\|_\infty,
    \end{equation}
    where $\tilde{g}_1$ is defined by
    \begin{equation}\label{tg1}
\tilde{g}_1(x)=g_1(x)+\left(g_2(x)-\tilde{g_2}(x)\right)\log |x-\alpha|,
\end{equation}
and
\begin{equation}
\tilde{g}_2(x)=\sum_{j=0}^{n_2-1}\frac{g_2^{(j)}(a)}{j!}(x-\alpha)^j.
\end{equation}

\end{theorem}
\begin{proof}
For any $ K \in C(I)$ with the form \eqref{e1}, it can be rewritten as
\[
K(x)=\tilde{g}_1(x)+\tilde{g}_2(x)\log|x-\alpha|.
\]
%
It is clear from the definition \eqref{tg1} that $\tilde{g}_1$ belongs to the Sobolev space $H^{n_2}(I)$.
Since the product $\tilde{g}_2(x)\log|x-\alpha|$ is in space $\aS_{n_1,n_2}^\alpha(I)$, there exists
\[
\bS_{n_1,n_2}^\alpha(\tilde{g}_2\log|\cdot-\alpha|)(x)=\tilde{g}_2(x)\log|x-\alpha|.
\]
 Then we have
    \begin{equation}\label{error}
\begin{split}
    \| K -\bS_{n_1,n_2}^\alpha  K \|_1& = \|\tilde{g}_1-\bS_{n_1,n_2}^\alpha\tilde{g}_1\|_1 \\
    &=\|\tilde{g}_1-\bS_{n_1,0}\tilde{g}_1+\bS_{n_1,n_2}^\alpha\bS_{n_1,0}\tilde{g}_1-\bS_{n_1,n_2}^\alpha\tilde{g}_1\|_1 \\
    &\leq (1+\|\bS_{n_1,n_2}^\alpha\|_{1,\infty})\|\tilde{g}_1-\bS_{n_1,0}\tilde{g}_1\|_1 \\
    &\leq 2(1+\|\bS_{n_1,n_2}^\alpha\|_{1,\infty})\|\tilde{g}_1-\bS_{n_1,0}\tilde{g}_1\|_\infty.
\end{split}
\end{equation}
The proof finishes.
\end{proof}

According to Theorem \ref{Inte-E1}, the interpolation error of the new interpolation is bounded by the norm of $S_{n_1,n_2}^\alpha$ and the polynomial interpolation error of a smoother function. It shows that the enriched singular functions in the basis can catch the singularity of the singular function and make the new interpolation behave like interpolating a smoother function by the polynomials. Hence this new interpolation is suitable for the singular functions.

While implementing the new interpolation method, we always interpolate on the Chebyshev points since it is more stable and accurate. Besides, the value of $n_2$ is always set to be smaller than that of $n_1$. It is because when $n_1$ and $n_2$ are compatible and large, the linear system \eqref{e5} could be unstable according to our numerical experience.
Hence, we next give a concrete error bound of the new interpolation based on the Chebyshev points with $n_2$ smaller than $n_1$.
To this end, we recall the error bound of the polynomial interpolation based on the Chebyshev points listed as a lemma \cite{XIANG2010chebyshev}.
\begin{lemma}\label{cheby}
If $f, f',\ldots,f^{(k-1)}$ are absolutely continuous on $[-1,1]$ and if $\int_{-1}^1 \frac{|f^{(k)}(t)|}{\sqrt{1-t^2}}dt=V_k<\infty$ for some $k\geq 1$, then for each $n\geq k+2$,
\[
\|f-\bS_{n,0}f\|_\infty\leq \frac{4V_k}{k\pi (n-1)\ldots(n-k)},
\]
where $\bS_{n,0}$ is based on $n$ Chebyshev points of the first kind.
\end{lemma}

Combining Theorem \ref{Inte-E1} and Lemma \ref{cheby}, we present the error bound for the new interpolation based on the Chebyshev points.
 \begin{corollary}\label{Inte-E}
    If $\bS_{n_1,n_2}^\alpha$ is defined based on a set of $n_1$ Chebyshev points of the first kind and the other $n_2$ distinct points with $n_1\geq n_2+2$, then the interpolation error of $\bS_{n_1,n_2}^\alpha$, for functions of the form
    $ K(x) =g_1(x)+g_2(x)\log|x-\alpha|$ with $g_1$ and $g_2$ smooth, satisfies
    \begin{equation}\label{Serror}
        \| K -\bS_{n_1,n_2}^\alpha  K  \|_1\leq\frac{8c(1+\|\bS_{n_1,n_2}^\alpha\|_{1,\infty})}{\pi n_2(n_1-1)\ldots(n_1-n_2)},
    \end{equation}
    where $c$ is a constant depending on $n_2$ and $\alpha$.
\end{corollary}
 \begin{proof}
    According to the proof of Theorem \ref{Inte-E1}, it is clear that the regularity of $\tilde{g}_1$ is the same as $(x-\alpha)^{n_2}\log|x-\alpha|$. Due to the integral $\int_{-1}^1 \frac{|\log|x-\alpha||}{\sqrt{1-x^2}}dx<\infty $, there exists a constant $c$ depending on $\alpha$ and $n_2$ such that
\[
\int_{-1}^1 \frac{|\tilde{g}^{(n_2)}(x)|}{\sqrt{1-x^2}}dx=c<\infty
\]
By Lemma \ref{cheby}, we obtain that
\begin{equation}\label{g1err}
\|\tilde{g}_1-\bS_{n_1,0}\tilde{g}_1\|_\infty\leq \frac{4c}{n_2\pi (n_1-1)\ldots(n_1-n_2)},
\end{equation}

The interpolation error \eqref{Serror} follows directly from the combination of the result of Theorem \ref{Inte-E1} and \eqref{g1err}.
 \end{proof}

If $n_1\gg n_2$ and $\|\bS_{n_1,n_2}^\alpha\|_{1,\infty}$ is bounded, it is known from Corollary \ref{Inte-E} that the interpolation error in the $L^1$-norm is about $\bO(n^{-n_2})$.

\section{New quadrature for singular integrals}
In this section, we develop a new quadrature based on the fast non-polynomial interpolation for the logarithmically singular integrals which have the form
\begin{equation}\label{integral}
\int_{-1}^1 K(x)dx
\end{equation}
where $K(x)=g_1(x)+\log|x-\alpha|g_2(x)$ with $g_1$ and $g_2$ smooth.
These integrals occur in a wide range of practical
applications such as electromagnetic and acoustic scattering problems in two dimensions and the analysis of wire antennas in three dimensions.
Sometimes, functions $g_1$ and $g_2$ of the integrand are mixed together in an expression and may be difficult to get their explicit expressions, or need much effort to evaluate each function,
such as the kernel function of the Hall\'{e}n's integral equation \cite{WANG2015wire}. The quadrature we propose here can deal directly with the integral without analyzing the functions $g_1$ and $g_2$ separately.

Our original idea to compute integrals of this kind can be decomposed into two steps. The first one is
to find an approximation of the form
\[
\tilde{g}_1(x)+\log|x-\alpha|\tilde{g}_2(x)
\]
 for the function $K$. This step shall be implemented by the non-polynomial interpolation method.
 The other step is to compute the integrals
\[
\int_{-1}^1 \tilde{g}_1(x)dx\;\text{ and}\; \int_{-1}^1 \log|x-\alpha|\tilde{g}_2(x)dx
 \]
analytically by using moments related to Chebyshev polynomials.
For this purpose, we need the computation of moments of the form
\[
\xi_k=\int_{-1}^1T_k(x)w(x)dx, \; k=0,1,\ldots,n-1
\]
for $w(x)=1$ or $w(x)=\log |x-\alpha|$.
Fortunately, both of these two type moments are known. To make the integration complete, we recall here the expressions for the moments. For $w(x)=1$, the moment, denoted as $\xi_{1,k}$, can be obtained by making change of variables $t=\cos x$,
\begin{equation}
   \xi_{1,k}=\begin{cases}\frac{2}{1-k^2}, \; k\; \text{is even},\\
    0,\; k\; \text{is odd}.
    \end{cases}
\end{equation}
When $w(x)=\log |x-\alpha|$, the moment, denoted as $\xi_{2,k}$, can be evaluated fast indirectly by the linear three-term recurrence \cite{DOMINGUEZ2014}. Let
\[
\eta_k:=\int_{-1}^1U_k(x)\log|x-\alpha|dx, k=0,1,\ldots,n-1,
\]
where $U_k$ is the Chebyshev polynomial of the second kind and of degree $k$, and set $\eta_{-1}=0$. It is known that
\[
\eta_0=(1-\alpha)\log(1-\alpha)+(1+\alpha)\log(1+\alpha)-2,
\]
and $\eta_k$ satisfies the following linear three-term recurrence
\[
\eta_k=\frac{2{\color{red}\alpha}k}{k+1}\eta_{k-1}-\frac{k-1}{k+1}\eta_{k-2}+\gamma_k
\]
where
\[
\gamma_k=\frac{2}{k+1}
\begin{cases}
(1-\alpha)\log(1-\alpha)+(1+\alpha)\log(1+\alpha)+\frac{2}{k^2-1},\;  \text{for even}\; k, \\
(1-\alpha)\log(1-\alpha)-(1+\alpha)\log(1+\alpha),\; \text{for odd}\; k.
\end{cases}
\]
Then the moment $\xi_{2,k}$ can be obtained through
\[
\xi_{2,0}=\eta_0,\; \xi_{2,k}=\frac{\eta_{k}-\eta_{k-2}}{2},\; k=1,2,\ldots,n-1.
\]
When $\alpha=\pm1$, the values for $\eta_0$ and $\gamma_k$ can be derived by taking the limit $\alpha\rightarrow \pm 1$.

Once we get the coefficients $\Ba$ and $\Bb$ of the decomposition of $K$, the new quadrature for $I[K]:=\int_{-1}^1 K(x)dx$ is given by
\[
I_{n_1,n_2}[K]:=\sum_{j=0}^{n_1-1}a_j\xi_{1,j}+ \sum_{j=0}^{n_2-1}b_j\xi_{2,j}.
\]
Thus the integral \eqref{integral} can be computed fast within about $\bO(n_2n\log n)$ operations.
Applying Corollary \ref{Inte-E} directly implies that the integral error for the new quadrature satisfies
    \begin{equation}
    \begin{split}
        |I[K]-I_{n_1,n_2}[K]|&=\left|\int_{-1}^1K(x)-\bS_{n_1,n_2}^\alpha K(x)dx \right| \leq \|K- \bS_{n_1,n_2}^\alpha K\|_1 \\
         &\leq \frac{8c(1+\|\bS_{n_1,n_2}^\alpha\|_{1,\infty})}{\pi n_2(n_1-1)\ldots(n_1-n_2)},
         \end{split}
    \end{equation}
    where $c$ is a constant depending on $n_2$ and $\alpha$.

\section{Numerical results}
We present in this section numerical examples to show the efficiency of the proposed new interpolation for functions with the logarithmic singularity and that of the new quadrature for singular integrals.

As the first example, we consider the non-polynomial interpolation of function
\[
K_1(x)=\sin(x)+e^x\log|x-\alpha|.
\]
Setting $\alpha=-1$, we show the dependence of the interpolation error on $n_1$ and $n_2$.
Let $e_{n_1,n_2}(K)$ denotes the numerical $L^1$ norm of the interpolation error computed by the Fej\'{e}r rule of the first kind on a graded mesh with 1024 points.  The Fej\'{e}r rule of the first kind is an interpolatory quadrature obtained by using the Chebyshev points as nodes and its performance is often comparable to Gaussian rules \cite{TREFETHEN2008}.
%
    We set $n=[4,8,16,32]$ with $n_2=1,2,3$ where $n=n_1+n_2$.
    As a comparison, we also apply the polynomial interpolation based on Chebyshev polynomials of degree less than $n_1$ to smoother functions related to $K_1$. The functions are
    \[
    \begin{split}
    S_1(x)&=\sin(x)+(e^x-e^\alpha)\log|x-\alpha|, \\
    S_2(x)&=\sin(x)+(e^x-e^\alpha-e^\alpha(x-\alpha))\log|x-\alpha|,\\
    S_3(x)&=\sin(x)+(e^x-e^\alpha-e^\alpha(x-\alpha)-e^\alpha(x-\alpha)^2/2)\log|x-\alpha|,
    \end{split}
    \]
    and $S_j\in H^j(I), j=1,2,3$. Their errors, denoted by $e_{n_1,0}(S_j)$ without ambiguity, are measured by the numerical maximum norm computed by sampling 100000 equally spaced points on$(-1,1]$. According to Theorem \ref{Inte-E1}, there exists $e_{n_1,j}(K)\leq 2(1+\|\bS_{n_1,j}^\alpha\|_{1,\infty})e_{n_1,0}(S_j)$.
    The result of the errors is shown in Table \ref{table1} which indicates that the errors of the new interpolation are much smaller than those of the polynomial interpolation for smoother functions. It is compatible with the theory.
    Besides, we also show the performance of the new interpolation when $n_1<n_2$ by approximating the same function $K_1$. Set $n=[4,8,16,32]$ with $n_1=1,2,3$ where $n_2=n-n_1$ and their errors are shown in Table \ref{table5}.
It shows that the new interpolation is valid for both the cases.
  \begin{table}[htb]
\begin{center}
\caption{Non-polynomial interpolation errors with $n_2<n_1$}
\label{table1}
\begin{tabular}{c|cc|cc|cc}
$n$ & $e_{n_1,1}(K_{1})$& $e_{n_1,0}(S_{1})$  & $e_{n_1,2}(K_{1})$ & $e_{n_1,0}(S_{2})$&  $e_{n_1,3}(K_{1})$ & $e_{n_1,0}(S_{3})$
\\
\hline
4&$3.0409e-02$&$1.6291e+00$&$1.1719e-01$&$2.8021e+00$&$4.7287e-02$&$1.4506e+00$\\
8&$4.7454e-04$&$4.1793e-02$&$1.3966e-04$&$2.2144e-02$&$1.9498e-03$&$8.2544e-02$\\
16&$3.1122e-05$&$7.7639e-03$&$6.6101e-07$&$1.5493e-04$&$3.0596e-08$&$7.7457e-06$\\
32&$1.9609e-06$&$1.6711e-03$&$9.9881e-09$&$6.2167e-06$&$1.0462e-10$&$4.5946e-08$\\
\end{tabular}
\end{center}
\end{table}

\begin{table}[htb]
\begin{center}
\caption{Non-polynomial interpolation errors with $n_2>n_1$}
\label{table5}
\begin{tabular}{c|ccc}
\hline
$n$ & $e_{1,n_2}(K_{1})$&  $e_{2,n_2}(K_{1})$ &   $e_{3,n_2}(K_{1})$ \\
\hline
4&$4.7287e-02$&$1.1719e-01$&$3.0409e-02$\\
8&$7.0331e-03$&$2.0282e-04$&$1.1268e-03$\\
16&$1.7808e-03$&$5.2863e-06$&$1.2492e-07$\\
32&$4.5226e-04$&$2.5225e-07$&$1.3404e-09$\\
\hline
\end{tabular}
\end{center}
\end{table}

For the second part of the first example, we explore the dependence of the interpolation error on the singular point, i.e. the dependence on $\alpha$. To this end, we let the values of $\alpha$ range over $[-1,1]$ and $n$ ranges  in [8 16 32]  with $n_2=2$.
\begin{figure}

  \begin{minipage}[t]{0.45\linewidth}
  \centering
  \includegraphics[width=3.5in]{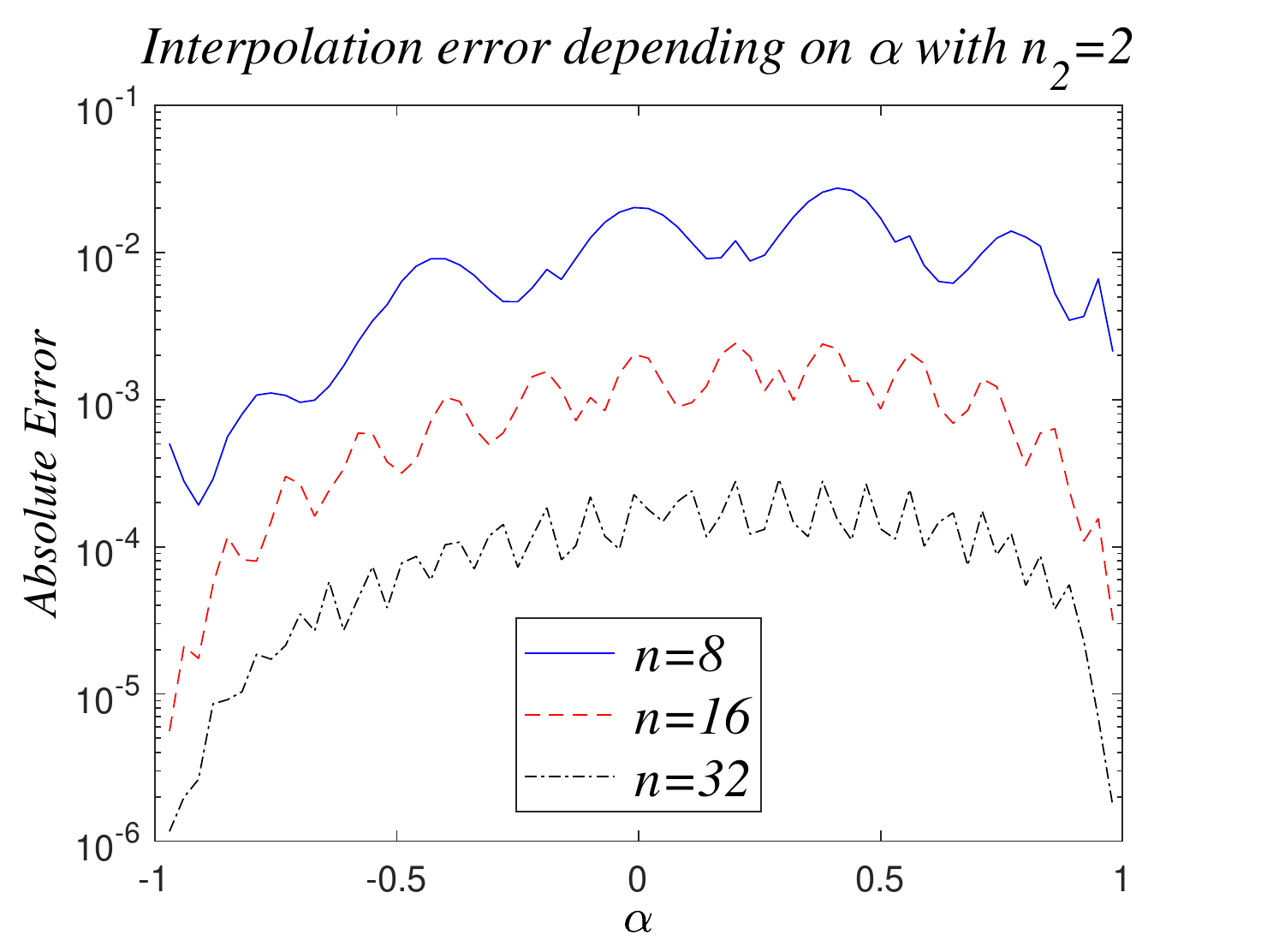}%
  \caption{Interpolation errors of function $K_1$ depending on $\alpha$}\label{p1}
  \end{minipage}%
  \hspace{0.2in}
  \begin{minipage}[t]{0.45\linewidth}
  \centering
  \includegraphics[width=3.5in]{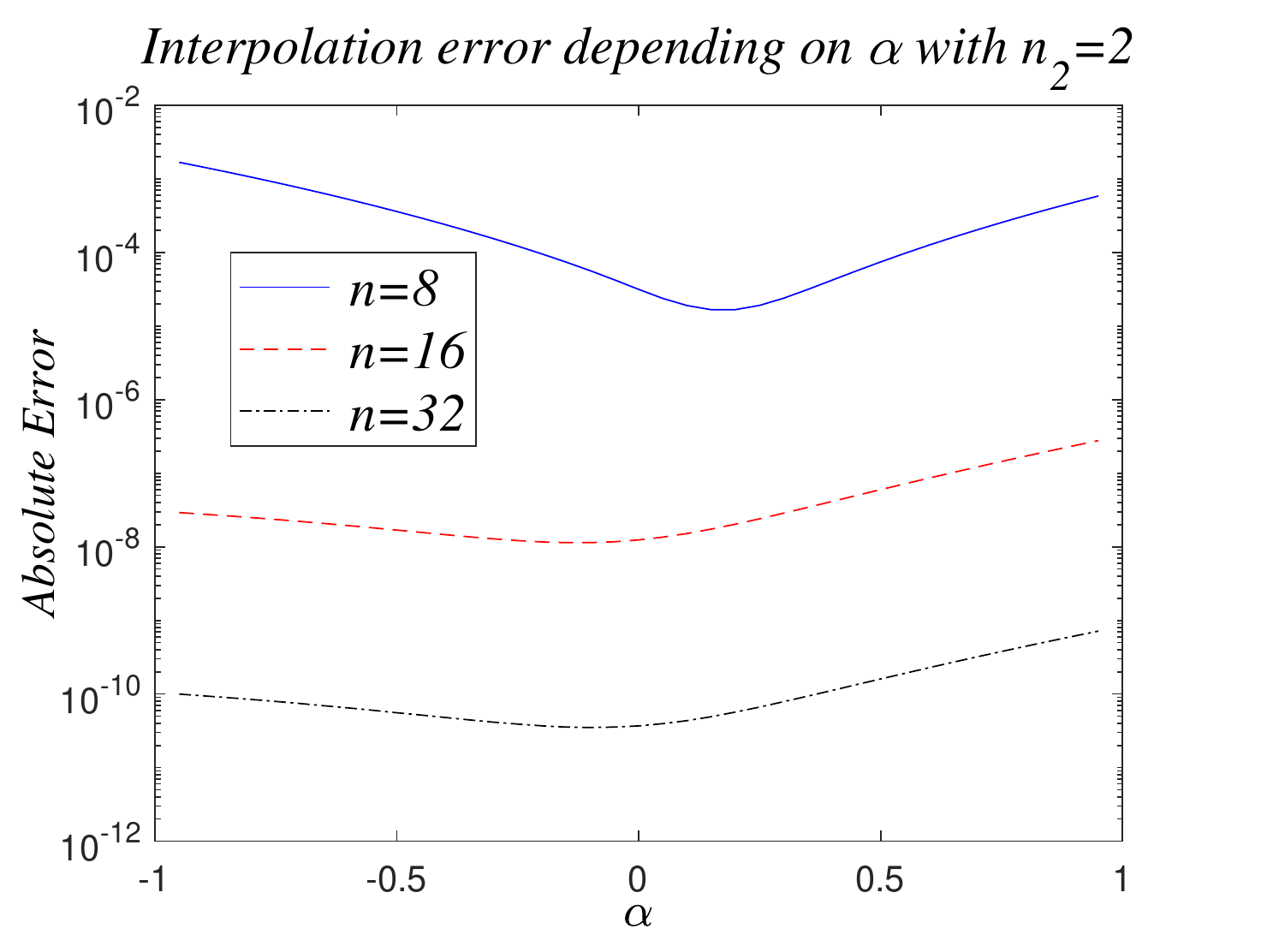}
  \caption{Interpolation errors of function $K_1$ depending on $\alpha$ by dividing the domain into two parts such that the singularity lies at the end}\label{p3}
  \end{minipage}%
\end{figure}
The dependence of the $L^1$ norm of the interpolation error on $\alpha$ is drawn in Figure \ref{p1}. It shows that the error of the new interpolation highly rely on the singular point $\alpha$.
When $\alpha$ is close to $\pm 1$, the interpolation error is much smaller. It is because that the interpolation points are clustered near $\pm 1$.
Though the interpolation is proposed for any $\alpha$ in $[-1,1]$, numerical results suggest that it is wise to separate the domain into two parts and then to transform each segment to $[-1,1]$ such that the singularity lies at one end.
We also present the dependence of the interpolation error in $L^1$ norm on the singular point $\alpha$ when combining the splitting technique in Figure \ref{p3}.

For the second example, we consider the interpolation of the functions which frequently occur in physical  applications, including
the Hankel function of the first kind and of zero order
    \[
    K_2(x)=H_0^{(1)}(x+1)
    \]
    and the kernel-related function of the Hall\'{e}n equation \cite{WANG2015wire}
    \[
    K_3(x)=\frac{1}{\pi}\int_0^{\pi/2} \frac{e^{-j2\beta\sqrt{(x+1)^2+\sin^2\phi'}}} {\sqrt{(x+1)^2+\sin^2\phi'}}d\phi',
    \]
$x\in(-1,1]$, $\beta=0.1$ and $j=\sqrt{-1}$. It is known that both of $K_2$ and $K_3$ have the form \eqref{e1} with $\alpha=-1$. Thus they can be theoretically interpolated well by the non-polynomial interpolation. The setting for the interpolation is $n=[4,8,16,32]$ and $n_2=[1,2,3]$. The result of interpolation errors for these two functions is presented in Table \ref{table6}. It is shown clearly that the non-polynomial interpolation is very effective for these two functions.
Besides, it shows that the interpolation error with $n_2=1$ is compatible with that of the case $n_2=2$ for both functions $K_2$ and $K_3$.
We note that the error depends closely on the characteristics of the ``coefficient function'' $g_2$.
For Hankel function $K_2$, it is found that when $x\rightarrow -1$, ``coefficient function'' $g_2$ of Hankel function is given by
\[
g_2(x)\approx\frac{2}{\pi}-\frac{1}{2\pi}(x+1)^2+\frac{1}{32\pi}(x+1)^4+\ldots.
\]
When appending the extra function $(x+1)\log(x+1)$ into the interpolation basis, it makes little contribution in approximating the Hankel function.
For $K_3$, it is known from \cite{BRUNO2007} that the
``coefficient function'' $g_2$ can be represented by an integral
\[g_2(x)=-\frac{1}{\pi}\int_0^\pi \frac{\cos\left(2\beta (x+1)\sin \phi'\right)}{\sqrt{1+(1+x)^2\cos^2\phi'}}d\phi' .\]
Since $g_2$ of $K_3$ is symmetric with respect to $x=-1$, it has the similar series as that of $K_2$.
It is why the interpolation errors with $n_2=2$ make no improvement for both $K_2$ and $K_3$ comparing with  the case of $n_2=1$.

\begin{table}[htb]
\begin{center}
\caption{Non-polynomial interpolation errors for functions $K_j, j=2,3$}
\label{table6}
\begin{tabular}{c|ccc|ccc}
$n$ & $e_{n_1,1}(K_{2})$ & $e_{n_1,2}(K_{2})$ &  $e_{n_1,3}(K_{2})$ &  $e_{n_1,1}(K_{3})$&  $e_{n_1,2}(K_{3})$&  $e_{n_1,3}(K_{3})$
\\
\hline
4&$3.7045e-03$&$6.4868e-03$&$1.9308e-03$&$6.6762e-04$&$2.0078e-03$&$1.0407e-03$\\
8&$1.7024e-05$&$1.4173e-05$&$6.2309e-06$&$5.8486e-06$&$6.6474e-06$&$3.6781e-05$\\
16&$2.2367e-07$&$1.5480e-07$&$2.6255e-10$&$1.1602e-07$&$7.9743e-08$&$3.2625e-09$\\
32&$3.3489e-09$&$2.2021e-09$&$1.4529e-13$&$1.7410e-09$&$1.1444e-09$&$8.2550e-13$\\
\end{tabular}
\end{center}
\end{table}


We next present numerical examples to show the efficiency for the new quadrature based on the non-polynomial interpolation.
For this purpose, we compute the integrals with logarithmically singular integrands:
\[
I_1:=\int_{-1}^1  (\sin x+e^x\log(x+1)) dx
\]
and
\[
 I_2(\alpha):=\int_{-1}^1 H_0^{(1)}(|x-\alpha|)dx,\;\text{with}\; \alpha=-1\;\text{and}\;1/4.
\]
The reference values are derived from Mathematica 10.0 with 15 exact digits.
To compute integrals $I_1$ and $I_2(\alpha)$, we set the number of quadrature nodes $n=[4,8,16,32]$ with $n_2=1,2,3$. According to the result of the first example, we do not calculate the integral $I_2(1/4)$ directly, but make a transformation such that the singularity locates at $-1$, i.e.
\[
I_2(\alpha)=\int_{-1}^1\left[\frac{1-\alpha}{2}H_0^{(1)} \left(\frac{1-\alpha}{2}(t+1)\right)+\frac{1+\alpha}{2}H_0^{(1)} \left(\frac{1+\alpha}{2}(t+1)\right)\right]dt.
\]
We note that the new integrand of $I_2(\alpha)$ still has the form \eqref{e1}.
 The absolute error between the approximation and the reference value is shown, respectively for $I_1$, $I_2(-1)$
 and $I_2(1/4)$, in Table \ref{table2}, \ref{table3} and \ref{table4}. For comparison, we also present the commonly used method, combining graded meshes with Fej\'{e}r rule of the first kind by transforming $[-1,1]$ to $[0,1]$.
 The number of segments is selected the same as $n$ and the number of quadrature nodes on each segment is chose to be 4. The graded parameter $q$ is set to be 4, i.e. the graded mesh is $\{(j/n)^4: j=0,1,\ldots,n\}$.
 It is obvious that the new quadrature for the logarithmically singular integrals is more accurate than the classical method and it needs much less evaluations of the integrands.
 In this example, the classical method needs four times of the evaluations of function compared with the new integration.
 To show the efficiency of the proposed method, we also present the comparison of CPU time between the new quadrature and the graded mesh method in computing $I_1$ and $I_2(-1)$ when attaching the same tolerance which is shown in Table \ref{table7} and \ref{table8}. The new quadrature method takes about $9\times 10^{-4}s$ while the classical method $1.5\times 10^{-3}s$. Numerical experiments show that the proposed quadrature is a competitive method in calculating the logarithmically singular integrals.

\begin{table}[htb]
\begin{center}
\caption{Absolute Errors of computation of $I_1$ }
\label{table2}
\begin{tabular}{c|cccc}
$n$ & $n_2=1$& $n_2=2$ & $n_2=3$ &  Graded mesh
\\
\hline
$4  $&$ 3.2523e-03 $&$  3.0721e-03  $&$ 2.9601e-04 $&$  2.5763e-04$\\
$8  $&$ 5.5618e-05 $&$  8.1836e-06  $&$ 3.7523e-04 $&$  1.8757e-05$\\
$16 $&$ 3.5207e-06 $&$  1.5837e-07  $&$ 9.9447e-09 $&$  1.8099e-06$\\
$32 $&$ 2.2078e-07 $&$  2.4433e-09  $&$ 3.5326e-11 $&$  1.7045e-07$\\
\end{tabular}
\end{center}
\end{table}

\begin{table}[htb]
\begin{center}
\caption{Absolute Errors of computation of $I_2(-1)$}
\label{table3}
\begin{tabular}{c|cccc}
$n$ & $n_2=1$ & $n_2=2$ & $n_2=3$ &  Graded mesh
\\
\hline
  $4 $&$7.3757e-04 $&$  1.7811e-03$&$   1.7071e-03 $&$  9.8617e-04$\\
  $8 $&$1.1963e-06 $&$  8.9449e-06$&$   3.3051e-05 $&$  9.1426e-05$\\
  $16 $&$2.1273e-08 $&$  1.4550e-07$&$   1.2065e-09 $&$  7.5273e-06$\\
  $32 $&$3.3892e-10 $&$  2.1459e-09$&$   8.0437e-13 $&$  5.8297e-07$\\
\end{tabular}
\end{center}
\end{table}

\begin{table}[htb]
\begin{center}
\caption{Absolute Errors of computation of $I_2(1/4)$}
\label{table4}
\begin{tabular}{c|cccc}
$n$ & $n_2=1$ & $n_2=2$ & $n_2=3$ &  Graded mesh
\\
\hline
4&$1.3000e-04$&$5.0916e-04$&$4.8218e-04$&$1.0289e-03$\\
8&$3.5141e-07$&$3.4277e-06$&$2.3821e-06$&$9.5176e-05$\\
16&$6.3122e-09$&$4.3286e-08$&$1.2532e-10$&$7.7819e-06$\\
32&$1.0061e-10$&$6.3712e-10$&$6.8883e-14$&$5.9920e-07$\\
\end{tabular}
\end{center}
\end{table}
\begin{table}[htb]
\begin{center}
\caption{CPU Time of computing $I_1$}
\label{table7}
\begin{tabular}{c|cccc}
 & $n_2=1$ & $n_2=2$ & $n_2=3$ &  Graded mesh
\\
\hline
$n$&$256$ & $64$ & $32$ & $ 256$\\
Error&$5.3959e-11$ &$3.8044e-11$ & $3.5326e-11$ & $8.8474e-11$\\
Time(s)&$9.2483e-04$ & $8.4059e-04$ & $8.5580e-04$ & $1.4313e-03$\\
\end{tabular}
\end{center}
\end{table}

\begin{table}[htb]
\begin{center}
\caption{CPU Time of computing $I_2(-1)$}
\label{table8}
\begin{tabular}{c|cccc}
 & $n_2=1$ & $n_2=2$ & $n_2=3$ &  Graded mesh
\\
\hline
$n$&$32$ & $64$ & $32$ & $ 256$\\
Error&$3.3892e-10$ &$3.3037e-11$ & $8.0437e-13$ & $2.2454e-10$\\
Time(s)&$8.2928e-04$ & $8.8050e-04$ & $8.8022e-04$ & $1.5378e-03$\\
\end{tabular}
\end{center}
\end{table}
\section{Acknowledgement}
This research is partially supported by the Construct Program of the
Key Laboratory of High Performance Computing and Stochastic Information Processing
and Natural Science Foundation of China under grant 11401207,11271370. The authors are grateful to anonymous referees for their helpful comments and useful suggestions for improvement of this paper.


\begin{thebibliography}{10}

\bibitem{BRUNO2007}
{\sc O.~Bruno and C.~Geuzaine}, {\em An integration scheme for
  three-dimensional surface scattering problems}, Journal of Computational and
  Applied Mathematics, 204 (2007), pp.~463--476.

\bibitem{2007BRUNO}
{\sc O.~P. Bruno and M.~C. Haslam}, {\em Regularity theory and superalgebraic
  solvers for wire antenna problems}, SIAM J. Sci. Comput., 29 (2007),
  p.~1375--1402.

\bibitem{CAO2003}
{\sc Y.~Cao, T.~Herdman, and Y.~Xu}, {\em A hybrid collocation method for
  {V}olterra integral equations with weakly singular kernels}, SIAM Journal on
  Numerical Analysis, 41 (2003), pp.~364--381.

\bibitem{CAO1994}
{\sc Y.~Cao and Y.~Xu}, {\em Singularity preserving {G}alerkin methods for
  weakly singular {F}redholm integral equations}, Journal of Integral Equations
  and Applications, 6 (1994), pp.~303--334.

\bibitem{CHENR2015}
{\sc R.~Chen and X.~Zhou}, {\em On quadrature of highly oscillatory integrals
  with logarithmic singularities}, Applied Mathematics and Computation, 265
  (2015), pp.~973--982.

\bibitem{COLTON2013}
{\sc D.~Colton and R.~Kress}, {\em Inverse Acoustic and Electromagnetic
  Scattering Theory}, Applied Mathematical Sciences, Springer, New York,
  3rd~ed., 2013.

\bibitem{CROW1993}
{\sc J.~Crow}, {\em Quadrature of integrands with a logarithmic singularity},
  Mathematics of Computation, 60 (1993), pp.~297--301.

\bibitem{DOMINGUEZ2014}
{\sc V.~Dominguez}, {\em {F}ilon-{C}lenshaw-{C}urtis rules for a class of
  highly-oscillatory integrals with logarithmic singularities}, Journal of
  Computational and Applied Mathematics, 261 (2014), pp.~299--319.

\bibitem{HASEGAWA2018}
{\sc T.~Hasegawa and H.~Sugiura}, {\em Uniform approximation to {C}auchy
  principal value integrals with logarithmic singularity}, Journal of
  Computational and Applied Mathematics, 327 (2018), pp.~1--11.

\bibitem{KANGH2018}
{\sc H.~Kang}, {\em Numerical integration of oscillatory airy integrals with
  singularities on an infinite interval}, J. Comp. Appl. Math., 333 (2018),
  pp.~314--326.

\bibitem{KANGH2015}
{\sc H.~Kang and C.~Ling}, {\em Computation of integrals with oscillatory
  singular factors of algebraic and logarithmic type}, J. Comp. Appl. Math.,
  285 (2015), pp.~72--98.

\bibitem{KANGH2017}
{\sc H.~Kang and J.~Ma}, {\em Quadrature rules and asymptotic expansions for
  two classes of oscillatory {B}essel integrals with singularities of algebraic
  or logarithmic type}, Appl. Numer. Math., 118 (2017), pp.~277--291.

\bibitem{RICE1969}
{\sc J.~R. Rice}, {\em On the degree of convergence of nonlinear spline
  approximation}, in Approximations with Special Emphasis on Spline Functions,
  I.~J. Schoenburg, ed., Academic, 1969.

\bibitem{TREFETHEN2008}
{\sc L.~N. Trefethen}, {\em Is {G}auss quadrature better than
  {C}lenshaw-{C}urtis?}, SIAM Review, 50 (2008), pp.~67--87.

\bibitem{WANGH2018}
{\sc H.~Wang}, {\em On the convergence rate of {C}lenshaw-{C}urtis quadrature
  for integrals with algebraic endpoint singularities}, J. Comp. Appl. Math.,
  333 (2018), pp.~87--98.

\bibitem{WANG2015wire}
{\sc Y.~Wang, J.~Luo, X.~Chen, and L.~Sun}, {\em A {C}hebyshev collocation
  method for {H}all\'{e}n's equation of thin wire antennas}, COMPEL: The
  International Journal for Computation and Mathematics in Electrical and
  Electronic Engineering, 34 (2015), pp.~1319 -- 1334.

\bibitem{XIANG2010chebyshev}
{\sc S.~Xiang, X.~Chen, and H.~Wang}, {\em Error bounds for approximation in
  {C}hebyshev points}, Numerische Mathematik, 116 (2010), pp.~463--491.

\end{thebibliography}

\end{CJK}
\end{document}